\documentclass[12pt, reqno]{amsart}
\usepackage{amsmath, amsthm, amscd, amsfonts, amssymb, graphicx, color}
\usepackage[bookmarksnumbered, colorlinks, plainpages]{hyperref}
\hypersetup{colorlinks=true,linkcolor=red, anchorcolor=green, citecolor=cyan, urlcolor=red, filecolor=magenta, pdftoolbar=true}

\newtheorem{theorem}{Theorem}[section]
\newtheorem{lemma}[theorem]{Lemma}
\newtheorem{proposition}[theorem]{Proposition}
\newtheorem{corollary}[theorem]{Corollary}
\theoremstyle{definition}
\newtheorem{definition}[theorem]{Definition}
\newtheorem{example}[theorem]{Example}

\theoremstyle{remark}
\newtheorem{remark}[theorem]{Remark}
\numberwithin{equation}{section}

\allowdisplaybreaks
\renewcommand{\leq}{\leqslant}
\renewcommand{\geq}{\geqslant}
\renewcommand{\epsilon}{\varepsilon}

\begin{document}
\sloppy

\setcounter{page}{1}

\title[]{Subrearrangement-invariant function spaces}\title[Subrearrangement-invariant function spaces]
{Subrearrangement-invariant function spaces}
\author[Ben Wallis]
{Ben Wallis}
\address 
{Division of Math/Science/Business, Kishwaukee College, Malta, IL 60150, United States.
} \email{
benwallis@live.com
}\





\begin{abstract}
Rearrangement-invariance in function spaces can be viewed as a kind of generalization of 1-symmetry for Schauder bases.  We define subrearrangement-invariance in function spaces as an analogous generalization of 1-subsymmetry.  It is then shown that every rearrangement-invariant function space is also subrearrangement-invariant.  Examples are given to demonstrate that not every function space on $(0,\infty)$ admits an equivalent subrearrangement-invariant norm, and that not every subrearrangement-invariant function space on $(0,\infty)$ admits an equivalent rearrangement-invariant norm.  The latter involves constructing a new family of function spaces inspired by D.J.H.\ Garling, and we further study them by showing that they are Banach spaces containing copies of $\ell_p$.
\end{abstract}
 
\maketitle

\section{Introduction}

A 1-symmetric basis can be viewed as a rearrangement-invariant function space over $\mathbb{N}$ (equipped with the counting measure).  Some bases, though, are only 1-subsymmetric, and we would like to explore an analogous property for certain nonatomic function spaces corresponding to this generalization.  For instance, in the case where we have a function space $X$ on $(0,\infty)$, we say that it is {\it subrearrangement-invariant} whenever the following holds:
\begin{quotation}\noindent For every function $f\in X$, every measurable $F\subseteq(0,\infty)$, and every strictly increasing bijection $m:(0,\infty)\to F$ such that $m$ and $m^{-1}$ are both measure-preserving, we have $\|f\circ m\|_X=\|f\textbf{1}_F\|_X$.\end{quotation}
Later, in \S2, we give a broader definition for other nonatomic function spaces besides $(0,\infty)$.

D.J.H.\ Garling was the first to prove that not every subsymmetric basis is symmetric, by publishing a counterexample in 1968 (\cite[\S5]{Ga68}).  Garling's sequence space, then, already furnishes us with a subrearrangement-invariant function space on $\mathbb{N}$ which is not rearrangement-invariant under any equivalent norm.  It is also well-known that there exist 1-unconditional bases which are not subsymmetric, and hence function spaces on $\mathbb{N}$ which fail to admit an equivalent subrearrangement-invariant norm.  In \S3, we extend these results to a purely nonatomic case, exhibiting an example of a function space on $(0,\infty)$ which is not essentially subrearrangement-invariant, and another example which is subrearrangement-invariant but not essentially rearrangement-invariant.  Some geometric properties of these ``Garling function spaces'' are then explored in \S4.

All Banach spaces and function spaces are taken over the real field $\mathbb{R}$, and all measure spaces we assume to be countably additive and $\sigma$-finite.   If $\theta$ and $\phi$ are real-valued functions, we use the symbolism $\theta(x)\approx_\epsilon\phi(y)$ to mean that for any $\epsilon>0$, the arguments $x$ and $y$ can be chosen such that
$$
\phi(y)-\epsilon<\theta(x)<\phi(y)+\epsilon.
$$
Beyond that, all notation and terminology is either standard (such as appears, for instance, in \cite{LT77}) or defined as encountered.  In \S5 we give an appendix with proofs of some measure-theoretic results we need for the main results, that are probably already known, but which we couldn't locate in the published literature.

\section{Subrearrangement-invariant function spaces}

For the following definition, $\beta$ denotes the Borel measure.

\begin{definition}Let $(\Omega,\mu)$ be a $\sigma$-finite measure space, and let $\mathcal{M}_0^+(\Omega)$ denote the cone of (nonnegative) $(\mu,\beta)$-measurable functions $f:\Omega\to[0,\infty]$.  Suppose $\rho:\mathcal{M}_0(\Omega)\to[0,\infty]$ satisfies the following properties for all $a\in(0,\infty)$ and all $f,g\in\mathcal{M}_0^+(\Omega)$:
\begin{itemize}
\item[(i)]  $\rho(f+g)\leq\rho(f)+\rho(g)$;
\item[(ii)]  $\rho(af)=a\rho(f)$; and
\item[(iii)]  $\rho(f)=0$ if and only if $f\equiv 0$ almost everywhere.
\end{itemize}
We can then define a normed linear space $(X,\|\cdot\|_X)$ consisting the a.e.-equivalence classes of measurable functions $f:\Omega\to[-\infty,\infty]$ satisfying $\|f\|_X:=\rho(|f|)<\infty$.  In this case we say that $\rho$ is a {\bf function norm} on $\Omega$, and $X$ is a {\bf function space} on $\Omega$ with respect to $\rho$.
\end{definition}

\noindent Note that our definition differs from other classes of function spaces such as Banach function spaces defined in \cite[\S1]{BS88} or K\"othe function spaces.  It is suitable for the present purposes, however.

\begin{remark}
If $(e_i)_{i=1}^\infty$ is a 1-unconditional basis for a Banach space $X$, we can define a function norm $\rho_X$ by setting, for all $f:\mathbb{N}\to[0,\infty]$,
$$\rho_X(f)=\left\{\begin{array}{ll}\displaystyle\left\|\sum_{i=1}^\infty f(i)e_i\right\|_X&\text{ if }\sum_{i=1}^\infty f(i)e_i\text{ converges, and}\\\\\infty&\text{ otherwise.}\end{array}\right.$$
In this way, $X$ can be viewed as a function space on $\mathbb{N}$, with respect to $(e_i)_{i=1}^\infty$.
\end{remark}

Let $(\Omega,\mu)$ be a $\sigma$-finite measure space, and $f:\Omega\to[-\infty,\infty]$ a $(\mu,\beta)$-measurable function (where $\beta$ is the Borel measure).  The {\bf distribution function} $\text{dist}_f:[0,\infty]\to[0,\infty]$ of $f$ is given by the rule
$$\text{dist}_f(s)=\mu\left\{x\in\Omega:|f(x)|>s\right\}.$$
Two such measurable functions $f$ and $g$ are said to be {\bf equimeasurable} whenever $\text{dist}_f=\text{dist}_g$. In this case we write $f\sim g$.

\begin{definition}
Let $(\Omega,\mu)$ be a $\sigma$-finite measure space.  A function space $X$ on $\Omega$ is called {\bf rearrangement-invariant} iff $\|f\|_X=\|g\|_X$ for all equimeasurable functions $f,g\in X$.  It is {\bf essentially rearrangement-invariant} provided it admits an equivalent rearrangement-invariant norm.
\end{definition}

\begin{remark}
The above definition follows \cite{BS88} rather than the somewhat more stilted definition of rearrangement-invariance found, for instance, in \cite{LT79}.
\end{remark}

Let $(E,\mu_E)$ and $(F,\mu_F)$ be measure spaces.  A map $m:E\to F$ is called {\bf $\boldsymbol{(\mu_E,\mu_F)}$-measurable} (or, when $\mu_E$ and $\mu_F$ are clear from context, simply, {\bf measurable}) if whenever $A$ is a measurable subset of $F$, the set $m^{-1}(A)$ is measurable in $E$.  The map $m$ is called a {\bf measure-preserving transformation} if whenever $A$ is a measurable subset of $F$, the set $m^{-1}(A)$ is measurable with $\mu_E(m^{-1}(A))=\mu_F(A)$.  If furthermore $m$ is bijective with $m^{-1}$ also measure-preserving, we say that it is a {\bf measure-isomorphism}.

\begin{definition}
If $E$ and $F$ are totally-ordered measure spaces, we denote by $\mathbb{MO}(E,F)$ the set of all maps $m:E\to F$ such that $m$ is both a stictly increasing measure-isomorphism.  Any such $m\in\mathbb{MO}(E,F)$ is called an {\bf$\boldsymbol{\mathbb{MO}}$-isomorphism} between $E$ and $F$.
\end{definition}

Let us now introduce the main subject under study.

\begin{definition}
Let $(\Omega,\mu)$ be a totally-ordered $\sigma$-finite measure space satisfying $\mu(\Omega)=\infty$.  We say that a function space $X$ on $\Omega$ is {\bf subrearrangement-invariant} if for every measurable $F\subseteq\Omega$, every $m\in\mathbb{MO}(\Omega,F)$, and every $f\in X$, we have $\|f\circ m\|_X=\|f$\emph{{\bf 1}}$_F\|_X$.  We say that $X$ is {\bf essentially subrearrangement-invariant} whenever it admits an equivalent subrearrangement-invariant norm.
\end{definition}

\noindent Here, the restriction $\mu(\Omega)=\infty$ has been included since $\mathbb{MO}(\Omega,F)$ would be empty otherwise, whenever $\mu(F)\neq\mu(\Omega)$, and that would make every function space on $\Omega$ trivially subrearrangement-invariant.

The following well-known fact is proved in the appendix.

\begin{proposition}\label{1-symmetric}
A 1-unconditional basis $(e_i)_{i=1}^\infty$ for a real Banach space $X$ is 1-symmetric if and only if $X$ is rearrangement-invariant as a function space on $\mathbb{N}$ with respect to $(e_i)_{i=1}^\infty$.  It is symmetric if and only if $X$ is essentially rearrangement-invariant.
\end{proposition}

Next, we give a result which in some sense justifies our definition of subrearrangement-invariance.

\begin{proposition}\label{1-subsymmetric}
A 1-unconditional basis $(e_i)_{i=1}^\infty$ for a real Banach space $X$ is 1-subsymmetric if and only if $X$ is subrearrangement-invariant as a function space on $\mathbb{N}$ with respect to $(e_i)_{i=1}^\infty$.  It is subsymmetric if and only if $X$ is essentially subrearrangement-invariant.
\end{proposition}

\begin{proof}
($\Rightarrow$):   Suppose $(e_i)_{i=1}^\infty$ is 1-subsymmetric.  Let $F\subseteq\mathbb{N}$ and $m\in\mathbb{MO}(\mathbb{N},F)$, and select any $f\in X$.  By 1-subsymmetry of $(e_i)_{i=1}^\infty$ we have
\begin{multline*}
\|f\circ m\|_X
=\left\|\sum_{i=1}^\infty(f)(m(i))e_i\right\|_X
=\left\|\sum_{i\in F}f(i)e_{m^{-1}(i)}\right\|_X
\\=\left\|\sum_{i\in F}f(i)\textbf{1}_F(i)e_{m^{-1}(i)}\right\|_X
=\left\|\sum_{i=1}^\infty f(i)\textbf{1}_F(i)e_i\right\|_X
=\|f\textbf{1}_F\|_X.
\end{multline*}
Hence, $X$ is subrearrangement-invariant with respect to $(e_i)_{i=1}^\infty$.

($\Leftarrow$):  Suppose now that $X$ is subrearrangement-invariant with respect to $(e_i)_{i=1}^\infty$.  Let $(e_{i_k})_{k=1}^\infty$ be a subsequence and $f\in X$.  Define $m(k)=i_k$ for $k\in\mathbb{N}$, and $F:=(i_k)_{k=1}^\infty$.  Clearly, $m\in\mathbb{MO}(\mathbb{N},F)$.  Define $g:\mathbb{N}\to[0,\infty]$ by letting $g(i)=(|f|\circ m^{-1})(i)$ if $i\in F$ and $g(i)=0$ otherwise.  We will need to check that $g\in X$, but this follows from the facts below, together with the identity $g\boldsymbol{1}_F=g$.  Now, by subrearrangement-invariance and 1-unconditionality we have
\begin{multline*}
\left\|\sum_{k=1}^\infty f(k)e_{i_k}\right\|_X
=\left\|\sum_{k=1}^\infty(f\circ m^{-1})(i_k)e_{i_k}\right\|_X
=\left\|\sum_{i\in F}(f\circ m^{-1})(i)e_{i}\right\|_X
\\=\left\|\sum_{i=1}^\infty g(i)\textbf{1}_F(i)e_{i}\right\|_X
=\|g\textbf{1}_F\|_X
=\|g\circ m\|_X
\\=\|f\circ m^{-1}\circ m\|_X
=\|f\|_X
=\left\|\sum_{i=1}^\infty f(i)e_i\right\|_X.
\end{multline*}

That $(e_n)_{n=1}^\infty$ is subsymmetric if and only if it is essentially subrearrangement-invariant follows easily by considering the equivalent norm $|||x|||=\sup_{(n_k)\in\mathbb{N}^\uparrow}\|\sum_{k=1}^\infty x_k^*(x)x_{n_k}\|$.
\end{proof}

It is well-known that every 1-symmetric basis is 1-subsymmetric.  Similarly, it is easy to show that rearrangement-invariance implies subrearrangement-invariance.  We just need a quick preliminary fact before we do.

\begin{proposition}[{\cite[Proposition 2.7.2]{BS88}}]\label{mpt-equimeasurable}
Let $m:E\to F$ be a measure preserving transformation between $\sigma$-finite measure spaces $(E,\mu_E)$ and $(F,\mu_F)$.  If $f:F\to[0,\infty]$ is a $(\mu_F,\beta)$-measurable function on $F$, then $f\circ m:E\to[0,\infty]$ is a $(\mu_E,\beta)$-measurable function on $E$, and $f$ and $f\circ m$ are equimeasurable.
\end{proposition}

\begin{proposition}\label{ri-is-sri}
Let $(\Omega,\mu)$ be a totally-ordered $\sigma$-finite measure space satisfying $\mu(\Omega)=\infty$.  If $X$ is a rearrangement-invariant function space on $\Omega$, then it is also subrearrangement-invariant.
\end{proposition}

\begin{proof}
Select any $f\in X$, measurable $F\subseteq\Omega$, and $m\in\mathbb{MO}(\Omega,F)$.   Notice that $f|_F\circ m=f\circ m$ so that, by Proposition \ref{mpt-equimeasurable}, $f\circ m\sim f|_F$.  We also clearly have $f\textbf{1}_F\sim f|_F$, and hence $f\circ m\sim f\textbf{1}_F$.  By rearrangment-invariance this means $\|f\circ m\|_X=\|f\textbf{1}_F\|_X$.
\end{proof}

Let us close this section by discussing the nontriviality of essential-subrearrangement invariance.  There are, after all, well-known examples of 1-unconditional bases which are not subsymmetric under any renorming, for instance the basis for the Tsirelson space.  This furnishes us with examples of function spaces on $\mathbb{N}$ which are not essentially subrearrangement-invariant.  The following example---a simple modification of the Schreier sequence space---gives us a function space on the purely nonatomic measure space $(0,\infty)$ which fails to be essentially subrearrangement-invariant.

\begin{example}
Denote by $\mathcal{A}$ the family of all subsets $A$ of $(0,\infty)$ satisfying $\lambda(A)\leq\sqrt{\inf A}$.  For a nonnegative $(\lambda,\beta)$-measurable function $f:(0,\infty)\to[0,\infty]$, we set
$$\rho_Y(f)=\sup_{A\in\mathcal{A}}\int_Af(t)\;dt.$$
Then $\rho_Y$ is a function norm, and we can denote by $Y$ the function space it generates.
Furthermore, $Y$ is a Banach space which fails to be essentially subrearrangement-invariant.
\end{example}

\begin{proof}
That $\rho_Y$ is a function norm is clear from the definition, and it's routine (via an argument such as in Proposition \ref{is-banach}) to show completeness.  So we need only prove that it fails to be essentially subrearrangement-invariant.  Select any $b\in(0,\infty)$.  When selecting $A\in\mathcal{A}$ to estimate $\|\boldsymbol{1}_{(0,b]}\|_Y$, we may assume without loss of generality that $\inf A\leq b$, else $\int_A\boldsymbol{1}_{(0,b]}(t)\;dt=0$.  Hence,
$$
\int_A\boldsymbol{1}_{(0,b]}(t)\;dt
\leq\lambda(A)
\leq\sqrt{\inf A}
\leq\sqrt{b}
$$
so that $\|\boldsymbol{1}_{(0,b]}\|_Y\leq\sqrt{b}$.  On the other hand, if $c\geq\sqrt{b}$ then
$
\|\boldsymbol{1}_{(c,c+b]}\|_Y
=b.
$
It is clear that $\boldsymbol{1}_{(c,c+b]}\circ m=\boldsymbol{1}_{(0,b]}$ for the shift map $m\in\mathbb{MO}((0,\infty),(c,\infty))$ defined by $m(t)=t+c$.  Hence,
$$
\frac{\|\boldsymbol{1}_{(c,c+b]}\|_Y}{\|\boldsymbol{1}_{(c,c+b]}\circ m\|_Y}
=\frac{\|\boldsymbol{1}_{(c,c+b]}\|_Y}{\|\boldsymbol{1}_{(0,b]}\|_Y}
\geq\frac{b}{\sqrt{b}}=\sqrt{b}.
$$
As $b\in(0,\infty)$ was arbitrary, it follows that $Y$ is not essentially subrearrangement-invariant.
\end{proof}

\section{Garling function spaces}

The converse of Proposition \ref{ri-is-sri} fails to hold in general, as can be seen from the following example.  If $1\leq p<\infty$ and $w=(w(k))_{k=1}^\infty$ is a nonincreasing sequence of positive real numbers satisfying $w\in c_0\setminus\ell_1$, then we can define the Garling sequence space $g(w,p)$ as the space of all scalar sequences $f:\mathbb{N}\to[-\infty,\infty]$ satisfying
$$
\|f\|_g:=\sup_{(i(k))_{k=1}^\infty\in\mathbb{N}^\uparrow}\left(\sum_{k=1}^\infty|f(i(k))|^pw(k)\right)^{1/p}
<\infty,
$$
where $\mathbb{N}^\uparrow$ denotes the family of all increasing sequences in $\mathbb{N}$.  (We usually also impose the condition that $w(1)=1$ but this is not always necessary.)  It is known from \cite[Proposition 2.4]{AAW18} and \cite[Lemma 3.1]{AALW18} that the unit vectors in $g(w,p)$ form a 1-unconditional basis which is 1-subsymmetric  but not symmetric.  In particular, thusly viewed as a function space on $\mathbb{N}$, by Propositions \ref{1-symmetric} and \ref{1-subsymmetric}, it is subrearrangement-invariant but fails to be rearrangement-invariant, or even just essentially rearrangement-invariant.  Nevertheless, it remains to be seen whether essential subrearrangement-invariance is a strictly weaker condition than essential rearrangement-invariance in the nonatomic setting. We devote this section, therefore, to exhibiting a function space on $(0,\infty)$ which is subrearrangement-invariant but fails to be essentially rearrangement-invariant.

To accomplish this, we shall simply generalize Garling's construction.  In fact, we will use the very same ``split into two sums'' trick that Garling did in his original paper \cite[\S5]{Ga68}.  However, in order for this strategy to work, we need to make some adaptations.  Part of that will involve the using the measure-theoretic results from \S2 of the present paper.  Also, we need to characterize Garling sequence spaces slightly differently.

\begin{proposition}
Fix a nonincreasing function $w:\mathbb{N}\to(0,\infty)$ with $w\in c_0\setminus\ell_1$.  For each function $f:\mathbb{N}\to[0,\infty]$ we define
$$\rho_g(f)=\sup_{\substack{E,F\subseteq\mathbb{N}\\m\in\mathbb{MO}(E,F)}}\left(\sum_{k\in E}(f\circ m)(k)^pw(k)\right)^{1/p}.$$
Then $\rho_g$ is a function norm generating the space $g(w,p)$.
\end{proposition}

\begin{proof}
Let $(i(k))_{k=1}^\infty\in\mathbb{N}^\uparrow$.  By taking $E=\mathbb{N}$, $F=(i(k))_{k=1}^\infty$, and $m(k)=i(k)$, it is clear that $\rho_g(f)\geq\|f\|_g$.  For the reverse inequality, let $E,F\subseteq\mathbb{N}$ and $m\in\mathbb{MO}(E,F)$.  We may assume without loss of generality that $E$ and $F$ are both infinite.  Thus, there is a unique $n\in\mathbb{MO}(\mathbb{N},E)$, and this satisfies $m\circ n\in\mathbb{MO}(\mathbb{N},F)$.  Since $w$ is nonincreasing, we have
\begin{multline*}
\sum_{k\in E}(f\circ m)(k)^pw(k)
=\sum_{j=1}^\infty(f\circ m\circ n)(j)^pw(n(j))
\\\leq\sum_{j=1}^\infty(f\circ m\circ n)(j)^pw(j)
\leq\|f\|_g^p.
\end{multline*}
  That $\rho_g$ is a function norm generating $g(w,p)$ follows immediately.
\end{proof}

\begin{definition}
Let $\mathbb{W}$ denote the set of all nonincreasing $(\lambda,\beta)$-measurable functions $W:(0,\infty)\to(0,\infty)$ satisfying the following conditions:
\begin{itemize}\item[(W1)]  $\displaystyle\lim_{t\to\infty}W(t)=0$,

\vspace{0.1cm}

\item[(W2)]  $\int_0^\infty W(t)\;dt=\infty$, and

\vspace{0.2cm}

\item[(W3)]  $\int_0^1W(t)\;dt<\infty$.
\end{itemize}
We also denote by $\lambda$ the Lebesgue measure and $\Lambda$ the family of Lebesgue-measurable subsets of $(0,\infty)$.  (Recall that $\beta$ is the Borel measure.)  For each $(\lambda,\beta)$-measurable $f:(0,\infty)\to[0,\infty]$, set
$$
\rho_G(f)=\sup_{\substack{E,F\in\Lambda\\m\in\mathbb{MO}(E,F)}}\left(\int_E(f\circ m)(t)^pW(t)\;dt\right)^{1/p},
$$
where $W\in\mathbb{W}$ and $1\leq p<\infty$.  We then define a {\bf Garling function space}, denoted $G_{W,p}(0,\infty)$, as the space of all a.e.-equivalence classes of measurable functions $f:(0,\infty)\to[-\infty,\infty]$ satisfying $\|f\|_G:=\rho_G(|f|)<\infty$.
\end{definition}

\begin{remark}
Conditions (W1) and (W2) are the only ones we use in \S3 and the proof of Proposition \ref{is-banach}.  However, for the other results in \S4, condition (W3) is needed.
\end{remark}

\noindent It is clear that $\rho_G$ is a function norm, and hence $G_{W,p}(0,\infty)$ is a function space on $(0,\infty)$.  We will show later in \S4 that it is in fact a Banach space, i.e.\ that it is complete.

\begin{proposition}
Fix $1\leq p<\infty$, and let $W\in\mathbb{W}$.  Then $G_{W,p}(0,\infty)$ is subrearrangement-invariant.
\end{proposition}

\begin{proof}
Fix $D\in\Lambda$ and $n\in\mathbb{MO}((0,\infty),D)$, and $f\in G_{W,p}(0,\infty)$.  Observe that there are $E,F\in\Lambda$ and $m\in\mathbb{MO}(E,F)$ such that
\begin{multline*}
\|f\textbf{1}_D\|_G^p
\approx_\epsilon\int_E((f\textbf{1}_D)\circ m)(t)^pW(t)\;dt
=\int_E(f\circ m)(t)^p(\textbf{1}_D\circ m)(t)W(t)\;dt
\\=\int_{m^{-1}(D)\cap E}(f\circ m)(t)^pW(t)\;dt
\leq\int_{m^{-1}(D)}(f\circ m)(t)^pW(t)\;dt
\\=\int_{m^{-1}(D)}((f\circ n)\circ(n^{-1}\circ m))(t)^pW(t)\;dt
\leq\|f\circ n\|_G^p,
\end{multline*}
where the last inequality follows from the fact that $n^{-1}\circ m$ is an $\mathbb{MO}$-isomorphism from $m^{-1}(D)$ onto its image.  On the other hand, there are $A,B\in\Lambda$ and $\ell\in\mathbb{MO}(A,B)$ such that
\begin{multline*}
\|f\circ n\|_G^p
\approx_\epsilon\int_A(f\circ n\circ\ell)(t)^pW(t)\;dt
=\int_A(f\circ n\circ\ell)(t)^p(\textbf{1}_D\circ n\circ\ell)(t)W(t)\;dt
\\=\int_A((f\textbf{1}_D)\circ(n\circ\ell))(t)^pW(t)\;dt
\leq\|f\textbf{1}_D\|_G^p,
\end{multline*}
where the first equality follows due to the fact that $\textbf{1}_D\circ n\circ\ell$ is the identity function on $A$, and the final inequality follows from the fact that $n\circ\ell$ is an $\mathbb{MO}$-isomorphism from $A$ onto its image.
\end{proof}

To show that a Garling function space fails to admit an equivalent rearrangement-invariant norm, we need the following intuitively obvious lemma.

\begin{lemma}\label{technical}
Fix $p\in[1,\infty)$ and $r\in(0,\infty)$.  Let $W\in\mathbb{W}$ and $f:(0,\infty)\to[0,\infty]$ a measurable function which is nondecreasing on $(0,r)$ and zero everywhere else.  Then there is $s\in[0,r]$ so that
$$
\|f\|_G=\left(\int_0^sf(t+r-s)^pW(t)\;dt\right)^{1/p}
$$
\end{lemma}

\noindent Unfortunately, it requires a somewhat technical proof.  We begin with some preliminaries.

\begin{proposition}[{\cite[Theorem 2.9.3]{Bo07}}]\label{distribution-integral}
Let $(\Omega,\mu)$ be a measure space and $f:\Omega\to[-\infty,\infty]$ a $(\mu,\beta)$-measurable function.  Then the $\mu$-integrability of $f$ is equivalent to the Lebesgue integrability of the function $t\mapsto\text{dist}_f(t)$, and
$$\int_\Omega|f|\;d\mu=\int_0^\infty\text{dist}_f(t)\;dt.$$
\end{proposition}

\begin{corollary}\label{mpt-integral}
If $E$ and $F$ are measurable subsets of $(0,\infty)$, and $f:(0,\infty)\to[0,\infty]$ is a (nonnegative) $(\lambda,\beta)$-measurable function, then for any measure-preserving transformation $m:E\to F$ we have
$$\int_E(f\circ m)(t)\;dt=\int_F f(t)\;dt.$$
\end{corollary}

\begin{proof}
By Proposition \ref{mpt-equimeasurable}, $f$ and $f\circ m$ are equimeasurable, which is to say that $\text{dist}_f=\text{dist}_{f\circ m}$.  Now by Proposition \ref{distribution-integral} we have
$$\int_E(f\circ m)(t)\;dt=\int_0^\infty\text{dist}_{f\circ m}(t)\;dt
=\int_0^\infty\text{dist}_f(t)\;dt=\int_F f(t)\;dt.$$
\end{proof}

The proof of the following is given in the appendix.

\begin{proposition}\label{mo-isomorphism}
Let $E$ be a measurable subset of $[-\infty,\infty]$ with $\lambda(E)<\infty$.  Then there is a measure-zero subset $E_0$ of $E$, a measure-zero subset $D_0$ of $[0,\lambda(E)]$, and an $\mathbb{MO}$-isomorphism between $E\setminus E_0$ and $[0,\lambda(E)]\setminus D_0$.
\end{proposition}

\begin{proof}[Proof of Lemma \ref{technical}]
First, observe that since the map
$$b\mapsto\int_0^bf(t+r-b)^pW(t)\;dt$$
is continuous on the compact set $[0,r]$, we can find $s\in[0,r]$ so that
\begin{equation}\label{4}
\int_0^sf(t+r-s)^pW(t)\;dt=\sup_{b\in[0,r]}\int_0^bf(t+r-b)^pW(t)\;dt.
\end{equation}
Let $E,F\in\Lambda$ and $m\in\mathbb{MO}(E,F)$ be such that
\begin{equation}\label{3}
\|f\|_G^p\approx_\epsilon\int_E(f\circ m)(t)^pW(t)\;dt.
\end{equation}
Without loss of generality we may assume that $F\subseteq(0,r)$, and set $b:=\lambda(F)\leq r$.  By Proposition \ref{mo-isomorphism} we can find measure-zero subsets $E_0$ of $E$ and $D_0$ of $(0,r)$, and an $\mathbb{MO}$-isomorphism
$$n:(0,b)\setminus D_0\to E\setminus E_0.$$

We claim that
\begin{equation}\label{1}(f\circ m\circ n)(t)^p\leq f(t+r-b)^p,\end{equation}
or, equivalently, $b-t\leq r-(m\circ n)(t)$, for each $t\in(0,b)\setminus D_0$.  Indeed, as $m\circ n$ is order-preserving, we have, for $c\in(t,b)\setminus D_0$,
$$(m\circ n)((t,c)\setminus D_0)\subseteq[(m\circ n)(t),(m\circ n)(c)]$$
and since $m\circ n$ is a measure isomorphism from $(0,b)\setminus D_0$ onto its image, we also have
\begin{multline*}
c-t
=\lambda(t,c)
=\lambda((t,c)\setminus D_0)
=\lambda((m\circ n)((t,c)\setminus D_0))
\\\leq\lambda[(m\circ n)(t),(m\circ n)(c)]
=(m\circ n)(c)-(m\circ n)(t).
\end{multline*}
For $\epsilon>0$ we are free to choose $c\in(t,b)\setminus D_0$ so that $b-c<\epsilon$.  Hence,
$$b-t<c-t+\epsilon\leq(m\circ n)(c)-(m\circ n)(t)+\epsilon\leq r-(m\circ n)(t)+\epsilon.$$
As $\epsilon>0$ was arbitrary, this means $b-t\leq r-(m\circ n)(t)$ as desired.

Next we claim that
\begin{equation}\label{2}(W\circ n)(t)\leq W(t),\end{equation}
or, equivalently, $t\leq n(t)$, for each $t\in(0,b)\setminus D_0$.  Indeed, for $\delta\in(0,t)\setminus D_0$ we have
$$n((\delta,t)\setminus D_0)\subseteq[n(\delta),n(t)]$$
and hence
$$t-\delta=\lambda(n((\delta,t)\setminus D_0))\leq\lambda[n(\delta),n(t)]=n(t)-n(\delta)\leq n(t).$$
As $\delta\in(0,t)\setminus D_0$ can be chosen arbitrarily close to zero, this means $t\leq n(t)$ as claimed.

From \eqref{1} and \eqref{2} we obtain that
$$(f\circ m\circ n)(t)^p(W\circ n)(t)\leq f(t-r+b)^pW(t)$$
for all $t\in(0,b)\setminus D_0$, and hence, by the above together with \eqref{4}, \eqref{3}, and Corollary \ref{mpt-integral}, we have
\begin{multline*}
\|f\|_G^p
\approx_\epsilon\int_E(f\circ m)(t)^pW(t)\;dt
=\int_{E\setminus E_0}(f\circ m)(t)^pW(t)\;dt
\\=\int_{(0,b)\setminus D_0}(f\circ m\circ n)(t)^p(W\circ n)(t)\;dt
\leq\int_{(0,b)\setminus D_0}f(t+r-b)^pW(t)\;dt
\\=\int_0^bf(t+r-b)^pW(t)\;dt
\leq\int_0^sf(t+r-s)^pW(t)\;dt
\leq\|f\|_G^p.
\end{multline*}
\end{proof}

We are now set to prove the main result of this section.

\begin{theorem}
If $W(t)=(t+1)^{-1/2}$ then $G_{W,1}(0,\infty)$ fails to admit an equivalent rearrangement-invariant norm.
\end{theorem}

\begin{proof}
Fix $r\in(0,\infty)$ and let $f_r:(0,\infty)\to[0,\infty]$ and $f_r^*:(0,\infty)\to[0,\infty]$ be defined by
$$
f_r(t)=\left\{\begin{array}{ll}(r+1-t)^{-1/2}&\text{ if }0<t<r,\\0&\text{ if }r\leq t<\infty\end{array}\right.
$$
and
$$
f_r^*(t)=\left\{\begin{array}{ll}(t+1)^{-1/2}&\text{ if }0<t<r,\\0&\text{ if }r\leq t<\infty.\end{array}\right.
$$
We claim that $f_r$ and $f_r^*$ are equimeasurable.  Indeed, it is clear that $\text{dist}_{f_r}(s)=\text{dist}_{f_r^*}(s)=r$ for all $0\leq s\leq(1+r)^{-1/2}$ and $\text{dist}_{f_r}(s)=\text{dist}_{f_r^*}(s)=0$ for all $1\leq s\leq\infty$.  Now select $(1+r)^{-1/2}<s<1$.  We have $f_r(t)>s$ if and only if both $0<t<r$ and $(r+1-t)^{-1/2}>s$, or, equivalently, $r+1-s^{-2}<t<r$.  In this case we have
$$\text{dist}_{f_r}(s)=\lambda\{t\in(0,\infty):f_r(t)>s\}=\lambda(r+1-s^{-2},r)=s^{-2}-1.$$
Similarly, $f_r^*(t)>s$ if and only if both $0<t<r$ and $(t+1)^{-1/2}>s$, or, equivalently, $0<t<s^{-2}-1$.  This gives us
$$\text{dist}_{f_r^*}(s)=\lambda\{t\in(0,\infty):f_r^*(t)>s\}=\lambda(0,s^{-2}-1)=s^{-2}-1$$
so that $f_r$ and $f_r^*$ are equimeasurable as claimed.

Note that
$$
\|f_r^*\|_G
\geq\int_0^r(t+1)^{-1}\;dt
=\log(r+1)\to\infty
$$
as $r\to\infty$.  Thus, to complete the proof, it is enough to show that $\|f_r\|_G$ is bounded by a number not depending on $r$.

Now we apply Garling's own ``split into two sums'' trick, except in our case the ``sums'' are actually integrals.  Since $f_r$ is increasing on its support $(0,r)$, and $W\in\mathbb{W}$, by Lemma \ref{technical} we must have $s\in[0,r]$ so that
\begin{multline*}
\|f_r\|_G
=\int_0^sf_r(t-s+r)W(t)\;dt.
=\int_0^s(1-t+s)^{-1/2}(t+1)^{-1/2}\;dt
\\=\int_0^{s/2}(1-t+s)^{-1/2}(t+1)^{-1/2}\;dt+\int_{s/2}^s(1-t+s)^{-1/2}(t+1)^{-1/2}\;dt.
\end{multline*}
Hence, it suffices to show that each of these pieces is bounded by a number not depending on $s$.  For the first piece, note that if $t\in(0,s/2]$ then $(1-t+s)^{-1/2}\leq(s/2+1)^{-1/2}$.  Hence,
\begin{multline*}
\int_0^{s/2}(1-t+s)^{-1/2}(t+1)^{-1/2}\;dt
\leq(s/2+1)^{-1/2}\int_0^{s/2}(t+1)^{-1/2}\;dt
\\=(s/2+1)^{-1/2}\cdot 2\left[(s/2+1)^{1/2}-2\right]
\leq 2.
\end{multline*}
For the second piece, note that if $t\in[s/2,s]$ then $(t+1)^{-1/2}\leq(s/2+1)^{-1/2}$, so that
\begin{multline*}
\int_{s/2}^s(1-t+s)^{-1/2}(t+1)^{-1/2}\;dt
\leq(s/2+1)^{-1/2}\int_{s/2}^s(1-t+s)^{-1/2}\;dt
\\=(s/2+1)^{-1/2}\cdot 2\left[(s/2+1)^{1/2}-2\right]
\leq 2.
\end{multline*}
\end{proof}

\section{Geometric properties of Garling function spaces}

In this section we show that Garling function spaces are Banach spaces containing $(1+\epsilon)$-isomorphic copies of $\ell_p$.  As a consequence, the space $G_{W,1}(0,\infty)$ is nonreflexive.  It remains an open question as to whether $G_{W,p}(0,\infty)$ is reflexive when $1<p<\infty$.

\begin{proposition}\label{is-banach}
Fix $1\leq p<\infty$ and $W\in\mathbb{W}$.  Then space $G_{W,p}(0,\infty)$ is complete.
\end{proposition}

\begin{proof}
Let $(f_i)_{i=1}^\infty$ be a Cauchy sequence in $G_{W,p}$.  Let $E,F\in\Lambda$ and $m\in\mathbb{MO}(E,F)$.  Observe that
$$
\|f_i-f_j\|_G^p
\geq\int_0^\infty|f_i(t)-f_j(t)|^pW(t)\;dt
\geq\||f_i|W^{1/p}-|f_j|W^{1/p}\|_{L_p(0,\infty)}^p
$$
so that $(|f_i|W^{1/p})_{i=1}^\infty$ is Cauchy in $L_p(0,\infty)$.  As such, it converges a.e.-pointwise to $g\in L_p(0,\infty)$.  Similarly,
\begin{align*}
\|f_i-f_j\|_G^p
&\geq\int_E|(f_i\circ m)(t)-(f_j\circ m)(t)|^pW(t)\;dt
\\&\geq\||f_i\circ m|W^{1/p}-|f_j\circ m|W^{1/p}\|_{L_p(E)}^p
\end{align*}
so that $(|f_i\circ m|W^{1/p})_{i=1}^\infty$ converges both in $L_p(E)$ and a.e.-pointwise to some $g_E\in L_p(E)$. Set $f:=gW^{-1/p}$ so that $(|f_i|)_{i=1}^\infty$ converges a.e.-pointwise to $f$.  As $(|f_i\circ m|W^{1/p})_{i=1}^\infty$ now converges a.e.-pointwise to $|f\circ m|W^{1/p}$, it follows that $|f\circ m|W^{1/p}$ and $g_E$ are a.e.-identical.

Since $(f_i)_{i=1}^\infty$ is Cauchy, we can find $M\in(0,\infty)$ so that $\|f_i\|_G^p\leq M$ for all $i\in\mathbb{N}$.  Furthermore, we can find $i_0\in\mathbb{N}$ so that $\|g_E-|f_{i_0}\circ m|W^{1/p}\|_{L_p(E)}^p\leq 1$.
Then
\begin{multline*}
\int_E|f\circ m|(t)^pW(t)\;dt
\leq\int_E|(|f|-|f_{i_0}|)\circ m|(t)^pW(t)\;dt+\int_E|f_{i_0}\circ m|(t)^pW(t)\;dt
\\=\|g_E-|f_{i_0}\circ m|W^{1/p}\|_{L_p(E)}^p+\int_E|f_{i_0}\circ m|(t)^pW(t)\;dt
\\\leq\|g_E-|f_{i_0}\circ m|W^{1/p}\|_{L_p(E)}^p+\|f_{i_0}\|_G^p
\leq 1+M.
\end{multline*}
As $E,F,m$ were arbitrary, we have $\rho_G(|f|)\leq(1+M)^{1/p}<\infty$ so that $f\in G_{W,p}(0,\infty)$.

Next, select $\epsilon>0$ and find $N\in\mathbb{N}$ so that $\|f_i-f_j\|_G<\epsilon/2$ for all $i,j\geq N$.  Select $j_0\geq N$ so that $\|g_E-|f_{j_0}\circ m|W^{1/p}\|_{L_p(E)}^p<\epsilon/2$. Then for $i\geq N$ we have
\begin{multline*}
\int_E|(f-f_i)\circ m|(t)^pW(t)\;dt
\leq\int_E||f|-|f_{j_0}||(t)^pW(t)\;dt+\int_E|f_i-f_{j_0}|(t)^pW(t)\;dt
\\\leq\|g_E-|f_{j_0}\circ m|W^{1/p}\|_{L_p(E)}^p+\|f_i-f_{j_0}\|_G^p
<\epsilon.
\end{multline*}
Again as $E,F,m$ were arbitrary and independent of $N$, it follows that $\|f-f_i\|_G^p<\epsilon$ for all $i\geq N$.  As $\epsilon>0$ was also arbitrary, $f_i\to f$ in $G_{W,p}(0,\infty)$.
\end{proof}

To close, we will show that when $1\leq p<\infty$ and $W\in\mathbb{W}$, the space $G_{W,p}(0,\infty)$ contains a copy of $\ell_p$.  To do this, we will use a basic sequence of characteristic functions as an auxiliary structure.  Let us gather some facts about it in the next lemma. In what follows, we denote $\boldsymbol{1}_i=\boldsymbol{1}_{(i-1,i]}$ for each $i\in\mathbb{N}$.

\begin{lemma}\label{characteristic-basis}
Fix $1\leq p<\infty$ and $W\in\mathbb{W}$, and set $K=\int_0^1W(t)\;dt$.   Then the sequence $(\boldsymbol{1}_i/K)_{i=1}^\infty$ is a normalized, monotone, 1-unconditional and 1-subsymmetric basic sequence in $G_{W,p}(0,\infty)$ which 1-dominates the unit vector basis $(g_i)_{i=1}^\infty$ of the Garling sequence space $g(w,p)$, where $w=(w(i))_{i=1}^\infty$ is formed by letting $w(i)=K^{-1}\int_{i-1}^iW(t)\;dt$ for each $i\in\mathbb{N}$.  Furthermore, they are isometrically equivalent for constant coefficients.
\end{lemma}

\begin{proof}
By replacing $W$ with $K^{-1}W$ if necessary, we may assume without loss of generality that $K=1$.

It's clear that $(\boldsymbol{1}_i)_{i=1}^\infty$ is normalized.  It is also clear that if $M<N\in\mathbb{N}$ and $(a_i)_{i=1}^\infty$ is any sequence of scalars then we have
$$
\left\|\sum_{i=1}^Ma_i\boldsymbol{1}_i\right\|_G
\leq\left\|\sum_{i=1}^Na_i\boldsymbol{1}_i\right\|_G,
$$
which is precisely the criterion for forming a monotone basic sequence.

Next we show that it is 1-unconditional.  Let $(a_i)_{i=1}^\infty,(b_i)_{i=1}^\infty\in c_{00}$ and satisfy $|a_i|\leq|b_i|$ for all $i\in\mathbb{N}$.  Then we can find $E,F\in\Lambda$ and $m\in\mathbb{MO}(E,F)$ such that, setting $U_j=m^{-1}(F\cap(j-1,j])$ for each $j\in\mathbb{N}$ so that $U_1<U_2<\cdots$ with $E=\bigcup_{j=1}^\infty U_j$ and each $\boldsymbol{1}_j\circ m=\boldsymbol{1}_{U_j}|_E$,
\begin{multline*}
\left\|\sum_{i=1}^\infty a_i\boldsymbol{1}_i\right\|_G^p
\approx_\epsilon\int_E\left|\sum_{i=1}^\infty a_i\boldsymbol{1}_i(m(t))\right|^pW(t)\;dt
=\int_E\sum_{i=1}^\infty|a_i|^p\boldsymbol{1}_i(m(t))W(t)\;dt
\\=\sum_{j=1}^\infty\int_{U_j}\sum_{i=1}^\infty|a_i|^p\boldsymbol{1}_i(m(t))W(t)\;dt
=\sum_{j=1}^\infty\int_{U_j}\sum_{i=1}^\infty|a_i|^p\boldsymbol{1}_{U_i}(t)W(t)\;dt
\\=\sum_{j=1}^\infty\int_{U_j}|a_j|^p\boldsymbol{1}_{U_j}(t)W(t)\;dt
\leq\sum_{j=1}^\infty\int_{U_j}|b_j|^p\boldsymbol{1}_{U_j}(t)W(t)\;dt
\end{multline*}
By an analogous argument we have
$$
\sum_{j=1}^\infty\int_{U_j}|b_j|^p\boldsymbol{1}_{U_j}(t)W(t)\;dt
=\int_E\left|\sum_{i=1}^\infty b_i\boldsymbol{1}_i(m(t))\right|^pW(t)\;dt
$$
whence
$$
\left\|\sum_{i=1}^\infty a_i\boldsymbol{1}_i\right\|_G^p
\leq\left\|\sum_{i=1}^\infty b_i\boldsymbol{1}_i\right\|_G^p
$$
so that $(\boldsymbol{1}_i)_{i=1}^\infty$ is 1-unconditional.

Let us show that it is 1-subsymmetric.  Indeed, if $(a_i)_{i=1}^\infty\in c_{00}$ and $(\boldsymbol{1}_{i_k})_{k=1}^\infty$ is some subsequence, then we can find $E,F\in\Lambda$ and $m\in\mathbb{MO}(E,F)$ such that
$$
\left\|\sum_{k=1}^\infty a_k\boldsymbol{1}_{i_k}\right\|_G^p
\approx_\epsilon\int_E\left|\sum_{k=1}^\infty a_k\boldsymbol{1}_{i_k}(m(t))\right|^pW(t)\;dt.
$$
Set $E'=\bigcup_{k=1}^\infty m^{-1}(F\cap(i_k-1,i_k])$, and define an $\mathbb{MO}$-isomorphism $\ell:(0,\infty)\to\bigcup_{k=1}^\infty(i_k-1,i_k]$ by gluing together the shift maps $(k-1,k]\mapsto(i_k-1,i_k]$.  Then $\ell^{-1}\circ m$ is an $\mathbb{MO}$-isomorphism between $E'$ and its image, and for each $k\in\mathbb{N}$ and $t\in E'$ we have $\boldsymbol{1}_{i_k}(m(t))=\boldsymbol{1}_k(\ell^{-1}(m(t)))$.  Furthermore, $\boldsymbol{1}_{i_k}(m(t))=0$ for each $k\in\mathbb{N}$ and $t\in E\setminus E'$.  Hence,
\begin{multline*}
\left\|\sum_{k=1}^\infty a_k\boldsymbol{1}_{i_k}\right\|_G^p
\approx_\epsilon\int_E\left|\sum_{k=1}^\infty a_k\boldsymbol{1}_{i_k}(m(t))\right|^pW(t)\;dt
=\int_{E'}\left|\sum_{k=1}^\infty a_k\boldsymbol{1}_{i_k}(m(t))\right|^pW(t)\;dt
\\=\int_{E'}\left|\sum_{k=1}^\infty a_k\boldsymbol{1}_k(\ell^{-1}(m(t)))\right|^pW(t)\;dt
\leq\left\|\sum_{k=1}^\infty a_k\boldsymbol{1}_k\right\|_G^p.
\end{multline*}
To show the reverse inequality, we instead choose $E,F,m$ so that
$$
\left\|\sum_{k=1}^\infty a_k\boldsymbol{1}_k\right\|_G^p
\approx_\epsilon\int_E\left|\sum_{k=1}^\infty a_k\boldsymbol{1}_k(m(t))\right|^pW(t)\;dt.
$$
Define $\ell$ as before so that $\ell\circ m$ is an $\mathbb{MO}$-isomorphism between $E$ and its image, and $\boldsymbol{1}_k(m(t))=\boldsymbol{1}_{i_k}(\ell(m(t)))$ for each $k\in\mathbb{N}$ and $t\in E$.  Then
\begin{multline*}
\int_E\left|\sum_{k=1}^\infty a_k\boldsymbol{1}_k(m(t))\right|^pW(t)\;dt
\\=\int_E\left|\sum_{k=1}^\infty a_k(\boldsymbol{1}_{i_k}(\ell(m(t)))\right|^pW(t)\;dt
\leq\left\|\sum_{k=1}^\infty a_k\boldsymbol{1}_{i_k}\right\|_G^p.
\end{multline*}
It follows that $(\boldsymbol{1}_i)_{i=1}^\infty$ is 1-subsymmetric.

To show that it 1-dominates $g(w,p)$, we again let $(a_i)_{i=1}^\infty\in c_{00}$.  Select any subsequence $(a_{i_k})_{k=1}^\infty$.  As before, there is an $\mathbb{MO}$-isomorphism $\ell:(0,\infty)\to\bigcup_{k=1}^\infty(i_k-1,i_k]$ defined by gluing together the shift maps $(k-1,k]\mapsto(i_k-1,i_k]$.  Note that $\boldsymbol{1}_{i_k}\circ\ell=\boldsymbol{1}_k$ for each $k\in\mathbb{N}$.  We now have
\begin{multline*}
\sum_{k=1}^\infty|a_{i_k}|^pw(k)
=\sum_{k=1}^\infty|a_{i_k}|^p\int_{k-1}^kW(t)\;dt
=\sum_{k=1}^\infty|a_{i_k}|^p\int_0^\infty\boldsymbol{1}_k(t)W(t)\;dt
\\=\int_0^\infty\sum_{k=1}^\infty|a_{i_k}|^p\boldsymbol{1}_k(t)W(t)\;dt
=\int_0^\infty\sum_{k=1}^\infty|a_{i_k}|^p\boldsymbol{1}_{i_k}(\ell(t))W(t)\;dt
\\=\int_0^\infty\left|\sum_{n=1}^\infty a_i\boldsymbol{1}_i(\ell(t))\right|^pW(t)\;dt
\leq\left\|\sum_{i=1}^\infty a_i\boldsymbol{1}_i\right\|_G^p
\end{multline*}
By taking the supremum over all subsequences we obtain
$$
\|(a_i)_{i=1}^\infty\|_g
\leq\left\|\sum_{i=1}^\infty a_i\boldsymbol{1}_i\right\|_G.$$

Finally, we consider the last part of the lemma, about being isometrically equivalent for constant coefficients to $(g_i)_{i=1}^\infty$.  Indeed, as $(\boldsymbol{1}_i)_{i=1}^\infty$ already 1-dominates it as shown above, we need only show the reverse inequality, i.e. that $(g_i)_{i=1}^\infty$ 1-dominates $(\boldsymbol{1}_i)_{i=1}^\infty$ for constant coefficients.  To that end, fix $N\in\mathbb{N}$ and let $E,F\in\Lambda$ and $m\in\mathbb{MO}(E,F)$ be such that
$$
\left\|\sum_{i=1}^N\boldsymbol{1}_i\right\|_G^p
\approx_\epsilon\int_E\sum_{i=1}^N\boldsymbol{1}_i(m(t))W(t)\;dt.
$$
For each $i=1,\cdots,N$, define $A_i:=m^{-1}(F\cap(i-1,i])$, and then set $A:=\bigcup_{i=1}^NA_i$.  It is clear that $\lambda(A)<\infty$, so by Proposition \ref{mo-isomorphism} we can find measure-zero subsets $D_0$ of $[0,\lambda(A)]$ and $A_0$ of $A$, and an $\mathbb{MO}$-isomorphism $n$ from $D:=[0,\lambda(A)]\setminus D_0$ onto $A\setminus A_0$.  We claim that $t\leq n(t)$ for all $t\in D$.  Indeed, if we set $b:=\inf n(D)$ then since $b\geq 0$ and $n(D_{\leq t})\subseteq[b,n(t)]$ we have
$$t=\lambda[0,t]=\lambda(D_{\leq t})=\lambda(n(D_{\leq t}))\leq\lambda[b,n(t)]=n(t)-b\leq n(t).$$
As $W$ is nonincreasing it follows that $W(n(t))\leq W(t)$ for all $t\in D$.  Note also that $\lambda(A)\leq N$ so that $D\subseteq[0,N]$.  Furthermore, it is clear that $\boldsymbol{1}_i(m(t))=0$ for all $i=1,\cdots,N$ and all $t\in E\setminus A$.  Together with Corollary \ref{mpt-integral} we now obtain
\begin{multline*}
\left\|\sum_{i=1}^N\boldsymbol{1}_i\right\|_G^p
\approx_\epsilon\int_E\sum_{i=1}^N\boldsymbol{1}_i(m(t))W(t)\;dt
=\int_A\sum_{i=1}^N\boldsymbol{1}_i(m(t))W(t)\;dt
\\=\sum_{j=1}^N\int_{A_j}\sum_{i=1}^N\boldsymbol{1}_i(m(t))W(t)\;dt
=\sum_{j=1}^N\int_{A_j}\boldsymbol{1}_j(m(t))W(t)\;dt
=\sum_{j=1}^N\int_{A_j}W(t)\;dt
\\=\int_AW(t)\;dt
=\int_DW(n(t))\;dt
\leq\int_DW(t)\;dt
\leq\int_0^NW(t)\;dt
=\sum_{k=1}^Nw(k)
=\left\|\sum_{i=1}^Ng_i\right\|_g^p.
\end{multline*}
As $(g_i)_{i=1}^\infty$ and $(\boldsymbol{1}_i)_{i=1}^\infty$ are both 1-subsymmetric, we are done.
\end{proof}

\begin{theorem}
Fix $1\leq p<\infty$ and let $W\in\mathbb{W}$.  Then for any $\epsilon>0$ the basic sequence $(\boldsymbol{1}_i)_{i=1}^\infty$ admits a normalized constant coefficient block basic sequence which is $(1+\epsilon)$-equivalent to $\ell_p$, and which is 2-complemented in $[\boldsymbol{1}_i]_{i=1}^\infty$.
\end{theorem}

\begin{proof}
Let $g(w,p)$, $(g_i)_{i=1}^\infty$, and $K$ be as in Lemma \ref{characteristic-basis}, so that $(\boldsymbol{1}_i/K)_{i=1}^\infty$ is isometrically equivalent to $(g_i)_{i=1}^\infty$ for constant coefficients.  It was shown in \cite[\S3]{AAW18} that there exists a constant coefficient block basic sequence of $(g_n)_{n=1}^\infty$ which is $(1+\epsilon)$-equivalent to $\ell_p$, for any $\epsilon>0$.  In particular, we can select
$$y'_i=\sum_{n=k_i}^{k_{i+1}-1}g_n\;\;\;\text{ and }\;\;\;y_i=\frac{y'_i}{\|y'_i\|_g}\;\;\;\text{ for each }i\in\mathbb{N},$$
where $1=k_1<k_2<k_3<\cdots\in\mathbb{N}$, so that $(y_i)_{i=1}^\infty$ is $(1+\epsilon)$-equivalent to $\ell_p$.

Next, write
$$x'_i=\sum_{n={k_i}}^{k_{i+1}-1}\boldsymbol{1}_n/K\;\;\;\text{ and }\;\;\;x_i:=\frac{x'_i}{\|x'_i\|_G}\;\;\;\text{ for each }i\in\mathbb{N},$$
where $1=k_1<k_2<k_3<\cdots\in\mathbb{N}$.

We claim that $(x_i)_{i=1}^\infty$ is 1-dominated by the unit vector basis of $\ell_p$.  Indeed, if $(a_i)_{i=1}^\infty\in c_{00}$ then we can find $E,F\in\Lambda$ and $m\in\mathbb{MO}(E,F)$ such that
\begin{multline*}
\left\|\sum_{i=1}^\infty a_ix_i\right\|_G^p
\approx_\epsilon \int_E\left|\sum_{i=1}^\infty a_ix_i(m(t))\right|^pW(t)\;dt
=\int_E\sum_{i=1}^\infty|a_i|^px_i(m(t))W(t)\;dt
\\=\sum_{i=1}^\infty|a_i|^p\int_Ex_i(m(t))W(t)\;dt
\leq\sum_{i=1}^\infty|a_i|^p\|x_i\|_G^p
=\sum_{i=1}^\infty|a_i|^p
\end{multline*}
so that $(x_i)_{i=1}^\infty\lesssim_1\ell_p$ as claimed.

By Lemma \ref{characteristic-basis}, $(\boldsymbol{1}_i/K)_{i=1}^\infty$ is isometrically equivalent to $(g_i)_{i=1}^\infty$ for constant coefficients, and so $\|y'_i\|_g=\|x'_i\|_G$ for each $i\in\mathbb{N}$. Again from Lemma \ref{characteristic-basis}, we know that $(g_i)_{i=1}^\infty$ is 1-dominated by $(\boldsymbol{1}_i/K)_{i=1}^\infty$.  It follows that
$$\ell_p\approx_{1+\epsilon}(y_i)_{i=1}^\infty\lesssim_1(x_i)_{i=1}^\infty\lesssim_1\ell_p.$$
That $(x_i)_{i=1}^\infty$ spans a 2-complemented subspace of $[\boldsymbol{1}_i]_{i=1}^\infty$ follows from the fact that constant-coefficient block basic sequences of a 1-subsymmetric basis are always 2-complemented (see, for instance, \cite[Proposition 3.a.4]{LT77}).
\end{proof}

\begin{remark}
Although $\ell_p$ is complemented in $[\boldsymbol{1}_i]_{i=1}^\infty$, we do not yet know if it is complemented in $G_{W,p}(0,\infty)$.
\end{remark}

\begin{corollary}
Fix $1\leq p<\infty$ and $W\in\mathbb{W}$.  Then for every $\epsilon>0$, the space $G_{W,p}(0,\infty)$ contains a subspace which is $(1+\epsilon)$-isomorphic to $\ell_p$.  Hence, in particular, the space $G(W,1)$ is nonreflexive.
\end{corollary}

\section{Appendix}

\begin{proposition}\label{bijection-lebesgue}
Let $E$ and $F$ be Lebesgue-measurable subsets of $\mathbb{R}$, and let $m:E\to F$ be a bijection which is both order-preserving and measure-preserving.  Then $m^{-1}$ is also order-preserving and measure-preserving, i.e.\ $m\in\mathbb{MO}(E,F)$.
\end{proposition}

\begin{proof}
Clearly, it is enough to show that $m^{-1}$ is measurable.  To that end, let us fix a measurable set $A\subseteq E$; we claim that $m(A)$ is also measurable, which will complete the proof.

Denote by $\mathcal{B}=\sigma(\tau)$ the Borel $\sigma$-algebra on $\mathbb{R}$, where $\tau$ denotes the usual metric topology $\mathbb{R}$.  Let $\tau_E$ be the subspace topology on $E$, i.e. the topology defined by
$$\tau_E=E\cap\tau:=\left\{E\cap U:U\in\tau\right\}.$$
Similarly, we denote by $\tau_F$ the subspace topology for $F$.  It is well-known (and easy to see) that the set
$$E\cap\mathcal{B}:=\{E\cap B:B\in\mathcal{B}\}$$
is a $\sigma$-algebra on $E$, called the {\it trace} $\sigma$-algebra.  Since $E\cap\tau\subset E\cap\mathcal{B}$, we obtain
$$\sigma(\tau_E)=\sigma(E\cap\tau)\subseteq\sigma(E\cap\mathcal{B})=E\cap\mathcal{B}.$$
For the reverse inclusion, define
$$\Sigma:=\{Y\subseteq\mathbb{R}:E\cap Y\in\sigma(\tau_E)\}.$$
It is routine to verify that $\Sigma$ is a $\sigma$-algebra on $\mathbb{R}$.  Also, it is clear that $\tau\subseteq\Sigma$, since for $U\in\tau$ we have $E\cap U\in\tau_E\subseteq\sigma(\tau_E)$.  It follows that $\sigma(\tau)\subseteq\Sigma$, whence also by definition of $\Sigma$ we obtain $E\cap\sigma(\tau)\subseteq\sigma(\tau_E)$.  This gives us the reverse inclusion as desired.  We now have the identity
$\sigma(\tau_E)=E\cap\mathcal{B},$
and an identical argument shows that
$\sigma(\tau_F)=F\cap\mathcal{B}.$

It's a well-known fact in real analysis that we can find $C\in\mathcal{B}$ such that $A\subseteq C$ and $\lambda(C\setminus A)=0$.  Now set $C'=E\cap C\in\sigma(\tau_E)$.  Since $\lambda(C\setminus A)=0$ there is a measure-zero set $D\in\mathcal{B}$ with $C'\setminus A\subseteq C\setminus A\subseteq D$.  Set $D':=E\cap D\in\sigma(\tau_E)$ so that $C'\setminus A\subseteq D'$ and $\lambda(D')=0$.

By a standard argument found, for instance, in the proof of \cite[Theorem 2.1.2]{Bo07}, it follows that $m(B)$ is Lebesgue-measurable whenever $B\in\sigma(\tau_E)$.  Thus we have $\lambda[m(D')]=0$.  Observe $m(C')\setminus m(A)=m(C'\setminus A)\subseteq m(D')$ so that (since subsets of measure-zero sets are themselves measure-zero) $\lambda[m(C')\setminus m(A)]=0$ as well.  Note also that since $C'\in\sigma(\tau_E)$ we have $m(C')$ measurable.  Since $A\subseteq C'$, we obtain $m(A)=m(C')\setminus[m(C')\setminus m(A)]$,
which shows that $m(A)$ is measurable.
\end{proof}

\begin{proposition}\label{bijection-exists}
Let $E$ and $F$ be Lebesgue-measurable subspaces of $[-\infty,\infty]$, and let $m:F\to E$ be a surjective measure-preserving transformation which is also order-preserving.  Then there is a measure-zero subset $F_0$ of $F$ such that $m$ is a bijection between $F\setminus F_0$ and $E$.
\end{proposition}

\begin{proof}
For each $x\in E$, let $I_x$ be an interval containing $m^{-1}\{x\}$ which is minimal under the relation $\subseteq$.  Since $m$ is order-preserving, the $I_x$'s are all disjoint, which means only countably many of them have positive measure.  In particular, $m^{-1}\{x\}$ is a singleton for all but countably many $x\in E$.  Set
$E_0:=\{x\in E:m^{-1}\{x\}\text{ is not a singleton}\}.$
For each $x\in E_0$, select some $f_x\in m^{-1}\{x\}$.  Now set
$F_0:=\bigcup_{x\in E_0}\left(m^{-1}\{x\}\setminus\{f_x\}\right).$
Clearly, $m$ is a bijection between $F\setminus F_0$ and $E$.  Observe that each $m^{-1}\{x\}\setminus\{f_x\}$ has measure zero and that $E_0$ is countable.  It follows that $F_0$ has measure zero.
\end{proof}

Let $E$ be a totally-ordered set.  An {\bf initial segment} of $E$ is any subset of $E'$ of $E$ such that $E'<E\setminus E'$.

\begin{proposition}\label{initial-segment-measure}
Let $E$ be a Lebesgue-measurable subset of $[-\infty,\infty]$ with $\lambda(E)<\infty$.  Then for each $t\in[0,\lambda(E)]$ there is an initial segment $E_t$ of $E$ such that $\lambda(E_t)=t$.
\end{proposition}

\begin{proof}
Note that if $E_t$ is an initial segment of $E\setminus\{-\infty,\infty\}$ then $E_t\cup\{-\infty\}$ is an initial segment of $E$ with the same measure as $E_t$.  Hence, without loss of generality, we may assume $E\subset\mathbb{R}$.  We may also assume that $E$ is bounded, since if the result holds in that case then it can be extended to the unbounded case by considering the union of sets
$$E_n=[-n,-n+1)\cap E\cap(n-1,n].$$

Say $E\subseteq[a,b]$ for $-\infty<a<b<\infty$.  Define $f:[a,b]\to[0,\lambda(E)]$ by the rule
$$f(x)=\lambda([a,x]\cap E).$$
Observe that if $y<x\in[a,b]$ then
$$|f(x)-f(y)|=\lambda((y,x]\cap E)\leq|x-y|$$
so that $f$ is Lipschitz, in particular, continuous.  As $f(a)=0$ and $f(b)=\lambda(E)$, we may now apply the Intermediate Value Theorem.
\end{proof}

\begin{lemma}\label{mo-transformation}
Let $E$ be a measurable subset of $\mathbb{R}$ with $\lambda(E)<\infty$.  For each $t\in[0,\lambda(E)]$, let $E_t$ be an initial segment of $E$ (whose existence is guaranteed by Proposition \ref{initial-segment-measure}).  Define the map $m:E\to[0,\lambda(E)]$ by the rule
$$ m (x)=\inf\{t\in[0,\lambda(E)]:x\in E_t\}.$$
Then $m$ is both measure-preserving and order-preserving.  Furthermore, $m$ can be extended to a map $m:\mathbb{R}\to[0,\lambda(E)]$ defined by
$$m(x)=\lambda((-\infty,x]\cap E).$$
\end{lemma}

\begin{proof}
It is obvious that $ m $ is order-preserving, and it is explicitly proved in \cite[Proposition 2.7.4]{BS88} that it is also measure-preserving.
For $x\in E$ we set $E_{\leq x}:=(-\infty,x]\cap E$, and observe that if $t>\lambda(E_{\leq x})$ then $x\in E_t$ and if $t<\lambda(E_{\leq x})$ then $x\notin E_t$.  It follows that $m(x)=\lambda(E_{\leq x})$ for all $x\in E$.  Thus, we can extend $m$ continuously to the function $M:\mathbb{R}\to[0,\lambda(E)]$ via the rule
$M(x)=\lambda((-\infty,x]\cap E).$
\end{proof}

\begin{proof}[Proof of Proposition \ref{mo-isomorphism}]
Since $\{-\infty,\infty\}$ has measure zero, we may assume without loss of generality that $E\subset\mathbb{R}$.
Let $m:\mathbb{R}\to E$ be as in Lemma \ref{mo-transformation}.  It is clear (as in, for instance, the proof of Proposition \ref{initial-segment-measure}) that $m$ is Lipschitz, and hence continuous in the usual sense as well.
Since $\lambda$ is inner-regular, we can find a sequence $(K_n)_{n=1}^\infty$ of compact sets and a measure-zero set $L$ such that
$E=L\cup\bigcup_{n=1}^\infty K_n.$
It is known that the image of a bounded measure-zero set under a Lipschitz function is again measure-zero.  Furthermore, the continuous image of a compact set is again compact, and in particular measurable.  We now have
$m(E)=m\left(L\cup\bigcup_{n=1}^\infty K_n\right)=m(L)\cup\bigcup_{n=1}^\infty m(K_n).$
It follows that $m(E)$ is measurable.
We can now apply Proposition \ref{bijection-exists} to find a subset $E_0$ of measure zero such that $m$ is a bijection between $E\setminus E_0$ and $m(E)$.  Set $D_0=[0,\lambda(E)]\setminus m(E)$.  We have $\lambda(D_0)=0$, and by Proposition \ref{bijection-lebesgue}, $m$ is an $\mathbb{MO}$-isomorphism between $E\setminus E_0$ and $[0,\lambda(E)]\setminus D_0$.
\end{proof}

\begin{proposition}\label{ryff}
If $f,g:\mathbb{N}\to[0,\infty)$ are equimeasurable with
$$\lim_{n\to\infty}f(n)=\lim_{n\to\infty}g(n)=0$$
then either they are both identically zero or else there is a measure-isomorphism $m:\text{supp}(f)\to\text{supp}(g)$ such that $g\circ m=f$ on $\text{supp}(f)$.
\end{proposition}

\begin{proof}
Obviously, if one of $f$ and $g$ is identically zero then, since they are equimeasurable, so is the other.  So let us assume that neither is identically zero.  Let $f^*(n)=\inf\{\lambda:\text{dist}_f(\lambda)\leq n\}$ denote the ``decreasing rearrangement'' of $f$.  Since $\lim_{n\to\infty}f(n)=0$ we have also $\lim_{n\to\infty}f^*(n)=0$.  Now \cite[Corollary 7.6]{BS88} gives us a measure-preserving transformation $m_f:\text{supp}(f)\to\text{supp}(f^*)$ such that $f=f^*\circ m_f$ on $\text{supp}(f)$, and analogously we get $m_g:\text{supp}(g)\to\text{supp}(g^*)$ with $g=g^*\circ m_g$.  Since $f$ and $g$ are equimeasurable, $f^*=g^*$, and hence $g=f^*\circ m_g$.  This means $g\circ m_g^{-1}=f^*$ and hence, setting $m=m_g^{-1}\circ m_f$, we obtain $g\circ m=g\circ m_g^{-1}\circ m_f=f^*\circ m_f=f$.
\end{proof}

\begin{proof}[Proof of Proposition \ref{1-symmetric}]
($\Rightarrow$):  Let $(e_i)_{i=1}^\infty$ be 1-symmetric, and suppose $f$ and $g$ are equimeasurable sequences in $X$.  Then so are $|f|$ and $|g|$.  If $f$ and $g$ are identically zero then $\|f\|_X=0=\|g\|_X$ and we are done.  Otherwise by Proposition \ref{ryff} there is a bijection $m:\text{supp}(f)\to\text{supp}(g)$ with $f=g\circ m$ on $\text{supp}(f)$.  Now we have, by 1-symmetry and 1-unconditionality
\begin{multline*}
\|f\|_X
=\left\|\sum_{i=1}^\infty f(i)e_i\right\|_X
=\left\|\sum_{i\in\text{supp}(f)}|f(i)|e_i\right\|_X
=\left\|\sum_{i\in\text{supp}(f)}|g(m(i))|e_i\right\|_X
\\=\left\|\sum_{i\in\text{supp}(g)}|g(i)|e_i\right\|_X.
=\left\|\sum_{i=1}^\infty g(i)e_i\right\|_X
=\|g\|_X.
\end{multline*}

($\Leftarrow$): Suppose that $X$ is rearrangement-invariant with respect to $(e_i)_{i=1}^\infty$, and select a permutation $\pi$ of $\mathbb{N}$.  Then its inverse $\pi^{-1}$ exists and is a measure-preserving transformation.  Select any $f\in X$, and note that $|f(i)|<\infty$ for all $i\in\mathbb{N}$. By Proposition \ref{mpt-equimeasurable}, $|f|$ and $|f|\circ\pi^{-1}$ are equimeasurable.  Now we have, by 1-unconditionality and rearrangement-invariance
\begin{multline*}
\left\|\sum_{i=1}^\infty f(i)e_{\pi(i)}\right\|_X
=\left\|\sum_{i=1}^\infty|f(i)|e_{\pi(i)}\right\|_X
=\left\|\sum_{i=1}^\infty(|f|\circ\pi^{-1})(i)e_i\right\|_X
\\=\||f|\circ\pi^{-1}\|_X
=\||f|\|_X
=\left\|\sum_{i=1}^\infty |f(i)|e_i\right\|_X
=\left\|\sum_{i=1}^\infty f(i)e_i\right\|_X.
\end{multline*}

\noindent That $(e_n)_{n=1}^\infty$ is symmetric if and only if it is essentially rearrangement-invariant is clear from considering the equivalent norm $|||x|||=\sup_{\sigma\in\Pi_\mathbb{N}}\|\sum_{n=1}^\infty e_n^*(x)e_{\sigma(n)}\|$.
\end{proof}

\noindent {\bf Acknowledgments.}  Thanks to Lukas Geyer and Ramiro Affonso de Tadeu Guerreiro for assisting in the proofs of Propositions \ref{bijection-lebesgue} and \ref{mo-isomorphism}, respectively.


\begin{thebibliography}{99}

\bibitem[AALW18]{AALW18}
Fernando Albiac, Jos\'e L. Ansorena, Denny Leung, and Ben Wallis.
``Optimality of the rearrangement inequality with applications to Lorentz-type sequence spaces,''
{\it Mathematical Inequalities and Applications} 21 (2018), pp127--132.

\bibitem[AAW18]{AAW18}
Fernando Albiac, Jos\'e L. Ansorena, and Ben Wallis.
``Garling sequence spaces,''
{\it Journal of the London Mathematical Society} 98:1 (2018 Apr 13), pp204--222.

\bibitem[Bo07]{Bo07}
Valdimir I. Bogachev.
{\it Measure Theory, Volume I}
(2007), ISBN 978-3-540-34513-8.

\bibitem[BS88]{BS88}
Colin Bennett and Robert Sharpley.
{\it Interpolation of Operators} (1988).
ISBN 978-0-120-88730-9.

\bibitem[Ga68]{Ga68}
D.J.H. Garling.
``Symmetric bases of locally convex spaces,''
{\it Studia Math.} 30 (1968), pp163--181.

\bibitem[LT77]{LT77}
Joram Lindenstrauss and Lior Tzafriri.
{\it Classical Banach Spaces I.}
(1977), ISBN 3-540-60628-9.

\bibitem[LT79]{LT79}
Joram Lindenstrauss and Lior Tzafriri.
{\it Classical Banach Spaces II.}
(1979), ISBN 3-540-08888-1.



\end{thebibliography}
\end{document}